\documentclass[a4paper,12pt]{article}

\usepackage{amsthm}
\usepackage{amsmath,amssymb,latexsym,amsfonts,mathrsfs}

\theoremstyle{plain}
\newtheorem{Thm}{Theorem}[section]

\newtheorem{Lem}[Thm]{Lemma}
\newtheorem{Prop}[Thm]{Proposition}
\newtheorem{Cor}[Thm]{Corollary}

\theoremstyle{definition}
\newtheorem{Def}[Thm]{Definition}
\newtheorem{Rem}[Thm]{Remark}

\newcommand{\eps}{\ensuremath{\varepsilon}}
\newcommand{\e}{{\rm e}} 
\renewcommand{\d}{{\rm d}} 

\newcommand{\tend}[2]{\mathrel{\mathop{\longrightarrow}\limits^{#1}_{#2}}}

\renewcommand{\hat}{\widehat}

\newcommand{\bN}{\ensuremath{\mathbb{N}}}

\newcommand{\bR}{\ensuremath{\mathbb{R}}}

\newcommand{\cC}{\ensuremath{\mathcal{C}}}

\newcommand{\cF}{\ensuremath{\mathcal{F}}}

\newcommand{\cS}{\ensuremath{\mathcal{S}}}

\newcommand{\sW}{\ensuremath{\mathscr{W}}}

\numberwithin{equation}{section}

\makeatletter
\renewcommand\section{\@startsection {section}{1}{\z@}%
                                   {-3.5ex \@plus -1ex \@minus -.2ex}%
                                   {2.3ex \@plus.2ex}%
                                   {\normalfont\large\bf}}
\makeatother

\makeatletter
\renewcommand\subsection{\@startsection {subsection}{1}{\z@}%
                                   {-3.5ex \@plus -1ex \@minus -.2ex}%
                                   {2.3ex \@plus.2ex}%
                                   {\normalfont\normalsize\bf}}
\makeatother

\makeatletter
\@addtoreset{footnote}{page}
\makeatother

\newcommand{\absol}[1]{\left| #1 \right|} 
\newcommand{\rbra}[1]{\left( #1 \right)} 
\newcommand{\cbra}[1]{\left\{ #1 \right\}} 
\newcommand{\sbra}[1]{\left[ #1 \right]} 

\renewcommand{\subitem}{\par\hangindent 5mm \hspace*{-2mm}}
\renewcommand{\subsubitem}{\par\hangindent 10mm \hspace*{2mm}}

\begin{document}

\begin{center}
{\Large \bf Wiener integral for the coordinate process 
under the $ \sigma $-finite measure unifying Brownian penalisations} 
\end{center}
\begin{center}
Kouji \textsc{Yano}\footnote{
Department of Mathematics, Graduate School of Science,
Kobe University, Kobe, JAPAN.}\footnote{
The research was supported by KAKENHI (20740060).}
\end{center}
\begin{center}
{\small \today}
\end{center}

\begin{abstract}
Wiener integral for the coordinate process 
is defined under the $ \sigma $-finite measure unifying Brownian penalisations, 
which has been introduced by Najnudel, Roynette and Yor (\cite{NRY1} and \cite{NRY2}). 
Its decomposition before and after last exit time from 0 
is studied. 
This study prepares for the author's recent study \cite{Y2} of Cameron--Martin formula 
for the $ \sigma $-finite measure. 
\end{abstract}



\section{Introduction}

Let $ \Omega = C([0,\infty ) \to \bR) $ and let $ \cF_{\infty } = \sigma(X_s:s \ge 0) $ 
where $ (X_s:s \ge 0) $ stands for the coordinate process. 
Let $ R^+ $ denote the law on $ (\Omega,\cF_{\infty }) $ 
of the 3-dimensional Bessel process starting from 0. 
We denote by $ R^- $ 
the law on $ (\Omega,\cF_{\infty }) $ of $ (-X_t) $ under $ R^+ $. 
We define $ R = \frac{R^+ + R^-}{2} $; 
in other words, 
$ R $ is the law on $ \Omega $ of $ (\eps X_t) $ 
under the product measure $ P(\d \eps) \otimes R^+(\d X) $ 
where $ P(\eps=1)=P(\eps=-1)=1/2 $. 
For $ u>0 $ and for two processes 
$ X^{(u)} = (X_t:0 \le t \le u) $ and $ Y = (Y_t: t \ge 0) $, 
we define the concatenation $ X^{(u)} \bullet Y $ 
as 
\begin{align}
(X^{(u)} \bullet Y)_t = 
\begin{cases}
X_t     \quad & \text{if} \ 0 \le t < u , \\
Y_{t-u} \quad & \text{if} \ t \ge u \ \text{and} \ X_u=Y_0 , \\
X_u       \quad & \text{if} \ t \ge u \ \text{and} \ X_u \neq Y_0 . 
\end{cases}
\label{}
\end{align}
We define the concatenation $ \Pi^{(u)} \bullet R $ 
as the law on $ (\Omega,\cF_{\infty }) $ 
of the concatenation $ X^{(u)} \bullet Y $ 
under the product measure $ \Pi^{(u)}(\d X^{(u)}) \otimes R(\d Y) $. 

In this paper we consider the following $ \sigma $-finite measure $ \sW $ 
defined on $ (\Omega,\cF_{\infty }) $: 
\begin{align}
\sW = \int_0^{\infty } \frac{\d u}{\sqrt{2 \pi u}} \rbra{ \Pi^{(u)} \bullet R } . 
\label{eq: def of sW}
\end{align}
This measure $ \sW $ 
has been introduced by Najnudel, Roynette and Yor 
(\cite{NRY1} and \cite{NRY2}; see also \cite{YYY}) 
in order to give a global view 
on various Brownian penalisations, 
which were developed by Roynette, Vallois and Yor 
(See \cite{RVY0}, \cite{RY} and the references therein). 

The purpose of this paper 
is to define and study {\em Wiener integral}, 
i.e., stochastic integral 
whose integrand is a deterministic function $ f $, written as 
\begin{align}
\int_0^{\infty } f(s) \d X_s 
\quad \text{under $ \sW $}. 
\label{eq: w-int}
\end{align}
We discuss its decomposition into the sum of two Wiener integrals 
before and after last exit time from 0; 
the former is for Brownian bridge, 
while the latter is for 3-dimensional Bessel process. 

As an application of the Wiener integral under $ \sW $, 
the author studies, in a separate paper \cite{Y2}, 
{\em Cameron--Martin formula} for $ \sW $, i.e., 
quasi-invariance of $ \sW $ under deterministic translations. 

Our main results are stated in the remainder of this section. 
All of their proofs will be given in Section \ref{sec: Wint sW}.

$ \mathbf{1^{\circ}).} $ 
{\bf Wiener integrals.} 
Let us recall definition of Wiener integrals. 
If the integrand is an indicator, we define 
\begin{align}
\int_0^{\infty } 1_{[t_1,t_2)}(s) \d X_s = X_{t_2}-X_{t_1} . 
\label{}
\end{align}
We extend it linearly so that 
we define Wiener integral $ \int_0^{\infty } f(s) \d X_s $ 
if the integrand $ f $ is a step function. 
In order to extend it 
to more general integrand functions, 
we need certain properties peculiar to the process considered; 
see below.

The following facts are well-known: 
If a sequence $ \{ f_n \} $ of step functions 
approximates $ f $ in $ L^2(\d s) $, then, 
for Brownian motion $ \{ (X_s),W \} $, it holds that 
\begin{align}
\int_0^{\infty } f_n(s) \d X_s \tend{}{n \to \infty } \int_0^{\infty } f(s) \d X_s 
\quad \text{in $ L^2(W) $} 
\label{}
\end{align}
and, for Brownian bridge $ \{ (X_s),\Pi^{(u)} \} $ with $ u>0 $, it holds that 
\begin{align}
\int_0^u f_n(s) \d X_s \tend{}{n \to \infty } \int_0^u f(s) \d X_s 
\quad \text{in $ L^2(\Pi^{(u)}) $} . 
\label{}
\end{align}
The symmetrized 3-dimensional Bessel process $ \{ (X_t),R \} $ 
requires an integrand function $ f $ to belong to $ L^2(\d s) $ 
and as well to the following class: 
\begin{align}
L^1 \rbra{ \frac{\d s}{\sqrt{s}} } = 
\cbra{ f : \int_0^{\infty } |f(s)| \frac{\d s}{\sqrt{s}} < \infty } . 
\label{}
\end{align}
Now the following fact holds (see Section \ref{sec: Wint 3Bes} for details): 
If a sequence $ \{ f_n \} $ of step functions 
approximates $ f $ both in $ L^2(\d s) $ and in $ L^1(\frac{\d s}{\sqrt{s}}) $, i.e., 
\begin{align}
\int_0^{\infty } |f_n(s)-f(s)|^2 \d s + \int_0^{\infty } |f_n(s)-f(s)| \frac{\d s}{\sqrt{s}} 
\to 0 , 
\label{eq: approx1}
\end{align}
then 
\begin{align}
\int_0^{\infty } f_n(s) \d X_s \tend{}{n \to \infty } \int_0^{\infty } f(s) \d X_s 
\quad \text{in $ R $-probability} . 
\label{}
\end{align}
Note that we cannot dispense with the assumption $ f \in L^1(\frac{\d s}{\sqrt{s}}) $; 
see Remark \ref{rem: uncentered}.

$ \mathbf{2^{\circ}).} $ 
{\bf Approximation theorem.} 
Let us discuss the Wiener integral \eqref{eq: w-int}. 
Note that the measure $ \sW $ on $ \cF_{\infty } $ is singular to Wiener measure 
and also to BES$ (3) $ measure; in fact, it holds $ \sW $-a.e. that 
$ |X_t| \to \infty $ as $ t \to \infty $ and that 
$ (X_t) $ takes both positive and negative values. 
Thus, in order to define the Wiener integral \eqref{eq: w-int}, 
we can only utilize the definition \eqref{eq: def of sW}. 

We need the following notion of convergence. 
We say that a sequence $ \{ Z_n \} $ of measurable functionals 
{\em converge locally in $ \sW $-measure} to a measurable functional $ Z $ 
if, for any $ \eps>0 $ and any measurable set $ A $ with $ \sW(A)<\infty $, 
\begin{align}
\sW \rbra{ A \cap \cbra{ \absol{ Z_n-Z } \ge \eps } } \to 0 
\quad \text{as} \ n \to \infty . 
\label{}
\end{align}
Thanks to the $ \sigma $-finiteness of $ \sW $, 
this notion of convergence plays a key role; 
see Section \ref{sec: loc} for details, and also \cite{BeGi}.

We introduce the following class of functions: 
\begin{align}
L^1 \rbra{ \frac{\d s}{1 + \sqrt{s}} } = 
\cbra{ f : \int_0^{\infty } |f(s)| \frac{\d s}{1 + \sqrt{s}} < \infty } . 
\label{}
\end{align}
Note that $ f \in L^1(\frac{\d s}{1 + \sqrt{s}}) $ 
if and only if $ f $ is locally integrable 
and $ \int_v^{\infty } |f(s)| \frac{\d s}{\sqrt{s}} < \infty $ for any (large) $ v>0 $. 

Wiener integral for $ X $ under $ \sW $ 
may be defined through the following theorem. 

\begin{Thm} \label{thm: approx}
Let $ f \in L^2(\d s) \cap L^1(\frac{\d s}{1 + \sqrt{s}}) $. 
Let $ \{ f_n \} $ be a sequence of step functions 
such that $ \{ f_n \} $ approximates $ f $ 
both in $ L^2(\d s) $ and in $ L^1(\frac{\d s}{1 + \sqrt{s}}) $, i.e., 
\begin{align}
\int_0^{\infty } |f_n(s)-f(s)|^2 \d s 
+ \int_0^{\infty } |f_n(s)-f(s)| \frac{\d s}{1 + \sqrt{s}} 
\to 0 . 
\label{eq: approx}
\end{align}
(Note that this condition is weaker than \eqref{eq: approx1}.) 
Then there exists a random variable, 
which will be denoted by $ \int_0^{\infty } f(s) \d X_s $, such that 
\begin{align}
\int_0^{\infty } f_n(s) \d X_s 
\tend{}{n \to \infty } 
\int_0^{\infty } f(s) \d X_s 
\quad \text{locally in $ \sW $-measure} . 
\label{}
\end{align}
The limit random variable does not depend 
up to $ \sW $-null sets 
on the choice of the approximating sequence $ \{ f_n \} $. 
\end{Thm}

Since $ L^1(\d s) \subset L^1(\frac{\d s}{1 + \sqrt{s}}) $, 
the following corollary is immediate from Theorem \ref{thm: approx}. 

\begin{Cor}
Let $ f \in L^2(\d s) \cap L^1(\d s) $. 
Let $ \{ f_n \} $ be a sequence of step functions 
such that $ \{ f_n \} $ approximates $ f $ both in $ L^2(\d s) $ and in $ L^1(\d s) $. 
Then it holds that 
\begin{align}
\int_0^{\infty } f_n(s) \d X_s 
\tend{}{n \to \infty } 
\int_0^{\infty } f(s) \d X_s 
\quad \text{locally in $ \sW $-measure} . 
\label{}
\end{align}
\end{Cor}

$ \mathbf{3^{\circ}).} $ 
{\bf Decomposition before and after last exit time.} 
Let $ g(X) $ denote the last exit time from 0 for $ X $: 
\begin{align}
g(X) = \sup \{ s \ge 0 : X_s = 0 \} . 
\label{}
\end{align}
For $ u \ge 0 $, let $ \theta_u X $ denote the shifted process: 
$ (\theta_u X)_s = X_{u+s} $, $ s \ge 0 $. 
Then the definition \eqref{eq: def of sW} says 
that the measure $ \sW $ may be described as follows: 
\subitem {\rm (i)} 
$ \displaystyle \sW(g(X) \in \d u) = \frac{\d u}{\sqrt{2 \pi u}} $; 
\subitem {\rm (ii)} 
For (Lebesgue) a.e. $ u \in [0,\infty ) $, 
it holds that, given $ g(X)=u $, 
\subsubitem {\rm (iia)} 
$ (X_s:s \le u) $ is a Brownian bridge from 0 to 0 of length $ u $; 
\subsubitem {\rm (iib)} 
$ ((\theta_u X)_s:s \ge 0) $ is 
a symmetrized 3-dimensional Bessel process. 

\noindent
Based on this path decomposition, we may also decompose the Wiener integral as follows. 

\begin{Thm} \label{thm: approx+}
Let $ f \in L^2(\d s) \cap L^1(\frac{\d s}{1 + \sqrt{s}}) $. 
Then it holds that 
\begin{align}
f(\cdot + u) \in L^1 \rbra{ \frac{\d s}{\sqrt{s}} } 
\quad \text{for a.e. $ u \in [0,\infty ) $} 
\label{}
\end{align}
and that there exists a jointly measurable functional $ (u,X) \mapsto I(f;u,X) $ such that 
\begin{align}
\int_0^{\infty } f(s) \d X_s 
= I(f;g(X),X) 
\quad \text{$ \sW $-a.e.}
\label{}
\end{align}
and that, for a.e. $ u \in [0,\infty ) $, 
\begin{align}
I(f;u,X) = 
\int_0^u f(s) \d X_s + \int_0^{\infty } f(s+u) \d (\theta_u X)_s 
\quad \text{$ (\Pi^{(u)} \bullet R) $-a.e.} 
\label{eq: expans}
\end{align}
In the right hand side of the expression \eqref{eq: expans}, 
the first and the second terms 
are Wiener integrals for the Brownian bridge $ (X_s:s \le u) $ 
and for the symmetrized 3-dimensional Bessel process $ ((\theta_u X)_s:s \ge 0) $, respectively. 
\end{Thm}

$ \mathbf{4^{\circ}).} $ 
{\bf Continuous modification.} 
Let $ 0<T<\infty $ be fixed and write $ \cF_T = \sigma(X_s : s \in [0,T]) $. 
Note that an $ \cF_T $-measurable set is $ \sW $-null 
if and only if it is $ W $-null. 
Although we cannot apply Radon--Nikodym theorem to $ \sW $ 
since it is not $ \sigma $-finite on $ \cF_T $, 
we have the following absolute continuity relationship (see \cite[equation (1.2.45)]{NRY2}): 
\begin{align}
\sW \sbra{ F_T(X) \e^{- g(X)} } 
= W \sbra{ F_T(X) \Lambda_T(X) } 
\label{}
\end{align}
for any bounded $ \cF_T $-measurable functional $ F_T(X) $, 
where $ \Lambda_T(X) $ is given as 
\begin{align}
\Lambda_T(X) 
= |X_T| \exp \rbra{ - g^{(T)}(X) } 
+ \int_0^{\infty } \frac{\d u}{\sqrt{2 \pi u}} \e^{-(T+u)} 
\exp \rbra{ - \frac{X_T^2}{2u} } 
\label{}
\end{align}
and where $ g^{(T)}(X) = \sup \{ s \le T : X_s=0 \} $. 
This shows that, if we assume that $ f \in L^2([0,T],\d s) $, 
there exists, under $ \sW $, a continuous modification 
$ \{ \int_0^t f(s) \d X_s : t \in [0,T] \} $ 
of $ \{ \int_0^{\infty } 1_{[0,t)} f(s) \d X_s : t \in [0,T] \} $. 
It is not, however, immediate from this fact 
whether there exists a jointly measurable functional 
which gives the decomposition of the continuous modification 
before and after the last exit time. 
The following theorem assures existence of such a functional. 

\begin{Thm} \label{thm: conti mod}
Let $ f \in L^2([0,T],\d s) $. Then there exists a jointly measurable functional 
$ (t,u,X) \mapsto I_t(f;u,X) $ such that the following statements hold: 
\subitem {\rm (i)} 
For a.e. $ u \in [0,\infty ) $ 
and for $ (\Pi^{(u)} \bullet R)(\d X) $-a.e. $ X $, 
the function $ t \mapsto I_t(f;u,X) $ is continuous; 
\subitem {\rm (ii)} 
For each $ t \in [0,T] $, 
\begin{align}
\int_0^t f(s) \d X_s = I_t(f;g(X),X) 
\quad \text{$ \sW $-a.e.;} 
\label{eq: cmod1}
\end{align}
\subitem {\rm (iii)} 
For each $ t \in [0,T] $ and for a.e. $ u \in [0,\infty ) $, 
\begin{align}
I_t(f;u,X) = 
\int_0^{u \wedge t} f(s) \d X_s + \int_0^{(t-u) \vee 0} f(s+u) \d (\theta_u X)_s 
\quad \text{$ (\Pi^{(u)} \bullet R) $-a.e.} 
\label{eq: cmod2}
\end{align}
\end{Thm}

$ \mathbf{5^{\circ}).} $ 
This paper is organized as follows. 
In Section \ref{sec: loc}, 
we study several properties 
of convergence a.e. and of convergence locally in measure, 
both considered on a $ \sigma $-finite measure space. 
In Section \ref{sec: prel}, 
we recall Wiener integrals for Brownian motion, Brownian bridge 
and 3-dimensional Bessel process. 
Section \ref{sec: Wint sW} 
is devoted to the proofs of the main theorems.

\section{Convergence locally in $ \sW $-measure} \label{sec: loc}

In this subsection, we only assume that 
$ (\Omega,\cF,\sW) $ is a $ \sigma $-finite measure space. 

\begin{Def} \label{def: loc}
Let $ Z,Z_1,Z_2,\ldots $ be $ \cF $-measurable functionals. 
As $ n \to \infty $, 
we say that {\em $ Z_n \to Z $ locally in $ \sW $-measure} 
if, for any $ \eps>0 $ and any $ A \in \cF $ with $ \sW(A)<\infty $, 
it holds that 
\begin{align}
\sW \rbra{ A \cap \cbra{ |Z_n-Z| \ge \eps } } \to 0 . 
\label{}
\end{align}
\end{Def}

Let us study some properties about this convergence. 
Define 
\begin{align}
L^1_+(\sW) = \cbra{ G:\Omega \to \bR_+, \ \text{$ \cF $-measurable}, \ 
\sW(G=0)=0 , \ \sW[G]<\infty } . 
\label{}
\end{align}
For $ G \in L^1_+(\sW) $, 
we define a probability measure $ \sW^G $ on $ (\Omega,\cF) $ by 
\begin{align}
\sW^G(A) = \frac{\sW[1_A G]}{\sW[G]} 
, \quad A \in \cF . 
\label{eq: sWG}
\end{align}
We obtain the following lemma, 
which only requires the $ \sigma $-finiteness of $ \sW $. 

\begin{Prop} \label{prop: loc in meas}
Let $ Z,Z_1,Z_2,\ldots $ be $ \cF $-measurable functionals. 
\subitem {\rm (i)} 
The following three statements are equivalent: 
\subsubitem {\rm (A1)} 
$ Z_n \to Z $ $ \sW $-a.e. 
\subsubitem {\rm (A2)} 
$ Z_n \to Z $ $ \sW^G $-a.s.~for some $ G \in L^1_+(\sW) $. 
\subsubitem {\rm (A3)} 
$ Z_n \to Z $ $ \sW^G $-a.s.~for any $ G \in L^1_+(\sW) $. 
\subitem {\rm (ii)} 
The following three statements are equivalent: 
\subsubitem  {\rm (B1)} 
$ Z_n \to Z $ locally in $ \sW $-measure. 
\subsubitem  {\rm (B2)} 
$ Z_n \to Z $ in $ \sW^G $-probability for some $ G \in L^1_+(\sW) $. 
\subsubitem  {\rm (B3)} 
$ Z_n \to Z $ in $ \sW^G $-probability for any $ G \in L^1_+(\sW) $. 
\subitem {\rm (iii)} 
$ Z_n \to Z $ locally in $ \sW $-measure if and only if 
one can extract, from an arbitrary subsequence, 
a further subsequence $ \{ n(k):k=1,2,\ldots \} $ along which 
$ Z_{n(k)} \to Z $ $ \sW $-a.e. 
\subitem {\rm (iv)} 
Convergence locally in $ \sW $-measure may be induced by 
some complete separable metric 
on the set of $ \cF_{\infty } $-measurable functionals. 
\end{Prop}

The reader may not be familiar with Claim (ii), 
so we give its proof for convenience of the reader, 
although it is an elementary argument. 

\begin{proof}[Proof of Claim (ii) of Proposition \ref{prop: loc in meas}]
Note that, since $ \sW $ is $ \sigma $-finite, 
we may take a family $ \{ E_m \} \subset \cF $ such that 
$ 0 < \sW(E_m) < \infty $ and $ \cup_m E_m = \Omega $. 

{[(B3) $ \Rightarrow $ (B1)]} 
Suppose that $ Z_m \to Z $ in $ \sW^G $-probability for any $ G \in L^1_+(\sW) $. 
Let $ A \in \cF $ such that $ \sW(A) < \infty $. 
We define $ G = 1_A + \sum_{m=1}^{\infty } 2^{-m} \sW(E_m)^{-1} 1_{E_m} $. 
Then we have $ G \in L^1_+(\sW) $, 
and consequently, we obtain $ Z_n \to Z $ in $ \sW^G $-probability. 
For any $ \eps>0 $, we have 
\begin{align}
\sW(A \cap \{ |Z_n-Z| \ge \eps \}) 
\le \sW[G] \sW^G(|Z_n-Z| \ge \eps) \to 0 . 
\label{}
\end{align}
Hence we obtain (B1). 

{[(B1) $ \Rightarrow $ (B2)]} 
Suppose that $ Z_n \to Z $ locally in $ \sW $-measure. 
Then, for any $ \eps>0 $, we have 
\begin{align}
\sum_{m=1}^{\infty } \frac{1}{2^m \sW(E_m)} \sW(E_m \cap \{ |Z_n-Z| \ge \eps \}) 
\tend{}{n \to \infty } 0 . 
\label{}
\end{align}
Now we define $ G = \sum_{m=1}^{\infty } 2^{-m} \sW(E_m)^{-1} 1_{E_m} $. 
Then we have $ G \in L^1_+(\sW) $ 
and obtain $ Z_n \to Z $ in $ \sW^G $-probability. 

{[(B2) $ \Rightarrow $ (B3)]} 
Suppose that $ Z_n \to Z $ in $ \sW^G $-probability for $ G \in L^1_+(\sW) $. 
Then, from any subsequence $ k(n) \to \infty $, 
we can extract a further subsequence $ k'(n) \to \infty $ 
along which $ Z_{k'(n)} \to Z $, $ \sW^G $-a.s., 
and consequently by {\rm (i)}, $ Z_{k'(n)} \to Z $, $ \sW $-a.e. 
Hence, by (A), we obtain (B3). 

The proof of Claim (ii) is now complete. 
\end{proof}

\section{Wiener integrals for Brownian motion, Brownian bridge and 3-dimensional Bessel process} \label{sec: prel}

\subsection{Wiener integrals} \label{sec: Wint}

Let $ \Omega = C([0,\infty ) \to \bR) $ 
and $ X=(X_t:t \ge 0) $ stand for the coordinate process. 
For $ 0 < u < \infty $, 
let $ \cS([0,u]) $ denote 
the set of all step functions $ f $ on $ [0,u] $ of the form: 
\begin{align}
f(t) = \sum_{k=1}^n c_k 1_{[t_{k-1},t_k)}(t) 
, \quad t \ge 0 
\label{eq: step func}
\end{align}
with $ n \in \bN $, 
$ c_k \in \bR $ ($ k = 1,\ldots,n $) 
and $ 0=t_0 < t_1 < \cdots < t_n < u $. 
Note that $ \cS([0,u]) $ is dense in $ L^2([0,u],\d s) $; 
indeed, we can take two sequences $ \{ f_n^+ \} $ and $ \{ f_n^- \} $ from $ \cS([0,u]) $ 
such that $ 0 \le f_n^+ \nearrow f \vee 0 $ 
and $ 0 \le f_n^- \nearrow (-f) \vee 0 $, 
and hence we see that $ f_n^+-f_n^- \to f $ in $ L^2([0,u],\d s) $. 
We write 
\begin{align}
\cS = \cbra{ f: \exists \ u \in [0,\infty ) \ \text{such that} \ 
f|_{[u,\infty )}=0 
\ \text{and} \ 
f|_{[0,u)} \in \cS([0,u]) } . 
\label{}
\end{align}
Then we know that $ \cS $ is dense in $ L^2(\d s) $. 

For a function $ f \in \cS $ and a process $ X $, 
we define 
\begin{align}
\int_0^{\infty } f(s) \d X_s =& \sum_{k=1}^n c_k (X_{t_k} - X_{t_{k-1}}) . 
\label{eq: pre Wiener integ}
\end{align}
For a more general function $ f $, 
we will write $ \int_0^{\infty } f(s) \d X_s $ 
for the limit of $ \int_0^{\infty } f_n(s) \d X_s $ in some sense 
with an approximating sequence $ \{ f_n \} \subset \cS $ of $ f $ 
in some functional space, 
and call it {\em Wiener integral of $ f $ for $ X $} 
whenever it is well-defined. 

In the following subsections, 
we give an introductory review on how to construct Wiener integrals for 
Brownian motion, Brownian bridge and 3-dimensional Bessel process, 
and on several facts about them. 

\subsection{Wiener integral for Brownian motion}

Denote $ \cF_{\infty }= \sigma (X_s : s \ge 0) $. 
Let $ W $ denote the Wiener measure, i.e., 
the law on $ (\Omega,\cF_{\infty }) $ of a (one-dimensional, standard) Brownian motion. 
Let us recall Wiener integral for the Brownian motion $ \{ (X_t),W \} $. 

The key to approximation is the following identity: 

\begin{Thm}[It\^o isometry] \label{thm: Itoiso W}
For any $ f \in \cS $, it holds that 
\begin{align}
W \sbra{ \absol{ \int_0^{\infty } f(s) \d X_s }^2 } 
= \int_0^{\infty } |f(s)|^2 \d s . 
\label{eq: Wint W}
\end{align}
\end{Thm}

Although it is widely known, 
we give the proof for completeness of the paper. 

\begin{proof}
Let $ f \in \cS $ be of the form \eqref{eq: step func}. 
Since $ \{ X_{t_k}-X_{t_{k-1}}: k=1,\ldots,n \} $ is an orthogonal system in $ L^2(W) $ 
and since $ W[(X_{t_k}-X_{t_{k-1}})^2] = t_k - t_{k-1} $, we have 
\begin{align}
W \sbra{ \absol{ \sum_{k=1}^n c_k (X_{t_k} - X_{t_{k-1}}) }^2 } 
= \sum_{k=1}^n c_k^2 \rbra{ t_k - t_{k-1} } . 
\label{}
\end{align}
This proves \eqref{eq: Wint W}. 
\end{proof}

The It\^o isometry \eqref{eq: Wint W} shows that 
if $ \{ f_n \} \subset \cS $ approximates $ f $ in $ L^2(\d s) $, then 
the Wiener integral $ \int_0^{\infty } f_n(s) \d X_s $ forms a Cauchy sequence in $ L^2(W) $ 
and hence it converges in $ L^2(W) $ 
where the limit random variable 
does not depend on the choice of the approximation sequence $ \{ f_n \} $. 

\begin{Def}
For $ f \in L^2(\d s) $, the Wiener integral $ \int_0^{\infty } f(s) \d X_s $ 
is defined as the $ L^2(W) $-limit of $ \int_0^{\infty } f_n(s) \d X_s $ 
for some sequence $ \{ f_n \} \subset \cS $ which approximates $ f $ in $ L^2(\d s) $. 
\end{Def}

\begin{Thm}
For any $ f \in L^2(\d s) $, the Wiener integral $ \int_0^{\infty } f(s) \d X_s $ 
satisfies the It\^o isometry \eqref{eq: Wint W} 
and is a centered Gaussian variable with variance $ \| f \|^2_{L^2(\d s)} $. 
\end{Thm}

\begin{proof}
The former assertion is immediate from Theorem \ref{thm: Itoiso W}. 
The latter is obvious via characteristic functions. 
\end{proof}

\subsection{Wiener integral for Brownian bridge}

Let $ 0<u<\infty $ be fixed. 
We write $ \Omega^{(u)} = C([0,u] \to \bR) $ 
and write $ X^{(u)}=(X_s: 0 \le s \le u) $ for the coordinate process. 
(We sometimes use the same symbol $ X^{(u)} $ to mean the part $ (X_s: 0 \le s \le u) $ 
of the coordinate process $ (X_s: s \ge 0) $ of $ \Omega = C([0,\infty ) \to \bR) $.) 
Denote $ \cF_u = \sigma(X_s : s \le u) $. 
We denote by $ \Pi^{(u)} $ the law on $ (\Omega^{(u)},\cF_u) $ of the {\em Brownian bridge}: 
\begin{align}
\Pi^{(u)}(\cdot) = W(\cdot|X_u=0) . 
\label{}
\end{align}
The process $ X^{(u)} $ under $ \Pi^{(u)} $ 
is a continuous centered Gaussian process with covariance 
$ \Pi^{(u)}[X_s X_t] = s-st/u $ for $ 0 \le s \le t \le u $. 
As a realization of $ \{ (X_s),\Pi^{(u)} \} $, we may take 
\begin{align}
\rbra{ B_s - \frac{s}{u} B_u : s \in [0,u] } 
\label{eq: realization}
\end{align}
where $ \{ (B_t),\Pi^{(u)} \} $ is a Brownian motion. 

Let us recall Wiener integral for Brownian bridge $ \{ X^{(u)},\Pi^{(u)} \} $ 
for $ 0<u<\infty $. 
(See also \cite{GW}.) 
For $ f \in L^2([0,u],\d s) $, we define $ \pi_u f \in L^2([0,u],\d s) $ by 
\begin{align}
(\pi_u f)(s) = f(s) - \frac{1}{u} \int_0^u f(t) \d t 
, \quad s \in [0,u] . 
\label{eq: piuf}
\end{align}
Note that $ \int_0^u \pi_u f(s) \d s = 0 $ and that 
\begin{align}
\| \pi_u f \|^2_{L^2([0,u],\d s)} = \| f \|^2_{L^2([0,u],\d s)} 
- \frac{1}{u} \rbra{ \int_0^u f(t) \d t }^2 . 
\label{}
\end{align}
In particular, we have 
\begin{align}
\| \pi_u f \|_{L^2([0,u],\d s)} \le \| f \|_{L^2([0,u],\d s)} . 
\label{eq: piuf le f}
\end{align}

\begin{Thm}[It\^o isometry for Brownian bridge] \label{thm: Itoiso Pi}
For any $ f \in \cS([0,u]) $, it holds that 
\begin{align}
\Pi^{(u)} \sbra{ \rbra{ \int_0^u f(s) \d X_s }^2 } = \int_0^u |\pi_uf(s)|^2 \d s . 
\label{eq: Wint Pi}
\end{align}
\end{Thm}

Although it is elementary, 
we give the proof for completeness of the paper. 

\begin{proof}
Let us adopt the realization \eqref{eq: realization} of the Brownian bridge. 
Taking the Wiener integrals for $ (B_t) $ of both sides of \eqref{eq: piuf}, we obtain 
\begin{align}
\int_0^u f(s) \d B_s - \frac{B_u}{u} \int_0^u f(s) \d s 
= \int_0^u (\pi_u f)(s) \d B_s . 
\label{}
\end{align}
This shows that 
\begin{align}
\int_0^u f(s) \d X_s = \int_0^u (\pi_u f)(s) \d B_s 
\label{eq: identity BB}
\end{align}
for all $ f \in \cS([0,u]) $. 
Thus we obtain \eqref{eq: Wint Pi} 
from the It\^o isometry \eqref{eq: Wint W} for Wiener integral for $ (B_t) $. 
\end{proof}

The It\^o isometry \eqref{eq: Wint Pi} and inequality \eqref{eq: piuf le f} show that 
if $ \{ f_n \} \subset \cS([0,u]) $ approximates $ f $ in $ L^2([0,u],\d s) $, then 
the Wiener integral $ \int_0^{\infty } f_n(s) \d X_s $ 
forms a Cauchy sequence in $ L^2(\Pi^{(u)}) $ 
and hence it converges in $ L^2(\Pi^{(u)}) $ 
where the limit random variable 
does not depend on the choice of the approximation sequence $ \{ f_n \} $. 

\begin{Def}
For $ f \in L^2([0,u],\d s) $, the Wiener integral $ \int_0^{\infty } f(s) \d X_s $ 
is defined as the $ L^2(\Pi^{(u)}) $-limit of $ \int_0^{\infty } f_n(s) \d X_s $ 
for some sequence $ \{ f_n \} \subset \cS $ which approximates $ f $ in $ L^2([0,u],\d s) $. 
\end{Def}

\begin{Thm}
For any $ f \in L^2([0,u],\d s) $, the Wiener integral $ \int_0^{\infty } f(s) \d X_s $ 
satisfies the It\^o isometry \eqref{eq: Wint Pi} 
and identity \eqref{eq: identity BB}, 
and is a centered Gaussian variable with variance $ \| \pi_u f \|^2_{L^2([0,u],\d s)} $. 
\end{Thm}

\begin{proof}
The former assertion is immediate from Theorem \ref{thm: Itoiso Pi}. 
The latter is obvious via characteristic functions. 
\end{proof}

\subsection{Wiener integral for 3-dimensional Bessel process via stochastic differential equation} 
\label{sec: Wint 3Bes}

Recall that $ R^+ $ is the law on $ (\Omega,\cF_{\infty }) $ 
of the 3-dimensional Bessel process starting from $ 0 $. 
It is well-known that $ R^+ $ is the law of the process $ (\sqrt{Z_t}) $ 
where $ (Z_t) $ is the unique strong solution 
to the stochastic differential equation 
\begin{align}
\d Z_t = 2 \sqrt{|Z_t|} \d \beta_t + 3 \d t 
, \quad Z_0 = 0 
\label{}
\end{align}
with $ (\beta_t) $ a Brownian motion. 
Under $ R^+ $, the process $ X $ satisfies 
\begin{align}
\d X_t = \d B_t + \frac{1}{X_t} \d t 
, \quad X_0 = 0 
\label{eq: 3BES SDE}
\end{align}
with a Brownian motion $ \{ (B_t),R^+ \} $. 

We may define Wiener integral for 3-dimensional Bessel process $ \{ (X_t),R^+ \} $ 
via the stochastic differential equation \eqref{eq: 3BES SDE}. 
Noting that 
\begin{align}
R^+ \sbra{ \frac{1}{X_t} } = \sqrt{ \frac{2}{\pi t} } , 
\label{}
\end{align}
we may give the following definition. 

\begin{Def}
For $ f \in L^2(\d s) \cap L^1(\frac{\d s}{\sqrt{s}}) $, 
we define 
\begin{align}
\int_0^{\infty } f(s) \d X_s 
= \int_0^{\infty } f(s) \d B_s + \int_0^{\infty } \frac{f(s)}{X_s} \d s . 
\label{eq: first def}
\end{align}
\end{Def}

Approximation by step functions is given as follows. 

\begin{Lem} \label{lem: approx of Wint 3Bes}
Let $ f \in L^2(\d s) \cap L^1(\frac{\d s}{\sqrt{s}}) $. 
Suppose that a sequence $ \{ f_n \} \subset \cS $ 
approximates $ f $ both in $ L^2(\d s) $ and in $ L^1(\frac{\d s}{\sqrt{s}}) $ 
(see \eqref{eq: approx1}). 
Then it holds that 
\begin{align}
\int_0^{\infty } f_n(s) \d X_s \tend{}{n \to \infty } \int_0^{\infty } f(s) \d X_s 
\quad \text{in $ L^1(R^+) $} . 
\label{}
\end{align}
\end{Lem}

\begin{proof}
Since $ f_n \to f $ in $ L^2(\d s) $, we have 
\begin{align}
\int_0^{\infty } f_n(s) \d B_s 
\tend{}{n \to \infty } 
\int_0^{\infty } f(s) \d B_s 
\quad \text{in $ L^2(R^+) $}, 
\label{}
\end{align}
and consequently, the convergence occurs also in $ L^1(R^+) $. 
Since 
\begin{align}
R^+ \sbra{ \int_0^{\infty } \frac{|f_n(s)-f(s)|}{X_s} \d s } 
= \sqrt{\frac{2}{\pi}} \int_0^{\infty } |f_n(s)-f(s)| \frac{\d s}{\sqrt{s}} 
\tend{}{n \to \infty } 0 , 
\label{}
\end{align}
we have 
\begin{align}
\int_0^{\infty } \frac{f_n(s)}{X_s} \d s 
\tend{}{n \to \infty } 
\int_0^{\infty } \frac{f(s)}{X_s} \d s 
\quad \text{in $ L^1(R^+) $}. 
\label{}
\end{align}
The proof is now complete. 
\end{proof}

\subsection{Wiener integral for 3-dimensional Bessel process via centered Bessel process}

There is another way of constructing Wiener integral for the 3-dimensional Bessel process, 
which is due to Funaki--Hariya--Yor (\cite{FHY2}; 
see also \cite{FHY1}, \cite{FHHY1} and \cite{FHHY2}). 
Note that this method will play a crucial role in \cite{Y2}. 

We define 
\begin{align}
\hat{X}_t = X_t - R^+[X_t] 
\label{}
\end{align}
and we call $ \{ (\hat{X}_t),R^+ \} $ {\em the centered 3-dimensional Bessel process}. 
For $ f \in \cS $ of the form \eqref{eq: step func}, 
the Wiener integral has already been defined as 
\begin{align}
\int_0^{\infty } f(s) \d \hat{X}_s = \sum_{k=1}^n c_k (\hat{X}_{t_k} - \hat{X}_{t_{k-1}}) . 
\label{}
\end{align}
We remark that neither the 3-dimensional Bessel process nor the centered one is Gaussian. 
So we cannot expect an isometry to hold 
similar to the It\^o isometries \eqref{eq: Wint W} and \eqref{eq: Wint Pi}. 
Funaki--Hariya--Yor \cite{FHY2} obtained the following remarkable inequality 
analogous to the It\^o isometries. 

\begin{Thm}[\cite{FHY2}] \label{thm: FHY2}
For any $ f \in \cS $ and 
any non-negative convex function $ \psi $ on $ \bR $, it holds that 
\begin{align}
R^+ \sbra{ \psi \rbra{ \int_0^{\infty } f(s) \d \hat{X}_s } } 
\le W \sbra{ \psi \rbra{ \int_0^{\infty } f(s) \d X_s } } . 
\label{eq: FHY ineq}
\end{align}
In particular, taking $ \psi(x)=x^2 $, one has 
\begin{align}
R^+ \sbra{ \absol{ \int_0^{\infty } f(s) \d \hat{X}_s }^2 } 
\le \int_0^{\infty } |f(s)|^2 \d s . 
\label{eq: Ito ineq}
\end{align}
\end{Thm}

For the proof of this Theorem, see \cite[Prop.4.1]{FHY2}. 

The inequality \eqref{eq: Ito ineq} shows that, 
if $ \{ f_n \} \subset \cS $ approximates $ f $ in $ L^2(\d s) $, then 
the Wiener integral $ \int_0^{\infty } f_n(s) \d \hat{X}_s $ forms a Cauchy sequence 
in $ L^2(R^+) $ and hence it converges in $ L^2(R^+) $ 
where the limit random variable 
does not depend on the choice of the approximation sequence $ \{ f_n \} $. 

\begin{Def}
For $ f \in L^2(\d s) $, the Wiener integral $ \int_0^{\infty } f(s) \d \hat{X}_s $ 
is defined as the $ L^2(R^+) $-limit of $ \int_0^{\infty } f_n(s) \d \hat{X}_s $ 
for some sequence $ \{ f_n \} \subset \cS $ which approximates $ f $ in $ L^2(\d s) $. 
\end{Def}

\begin{Thm} \label{thm: FHY}
For any $ f \in L^2(\d s) $ and 
any non-negative convex function $ \psi $ on $ \bR $, 
the inequality \eqref{eq: FHY ineq} remains valid, 
and so does \eqref{eq: Ito ineq}, in particular. 
\end{Thm}

\begin{proof}
We may take an approximation sequence $ \{ f_n \} $ of $ f $ in $ L^2(\d s) $ so that 
$ \sigma_n := \| f_n \|_{L^2(\d s)} $ converges increasingly 
to $ \sigma := \| f \|_{L^2(\d s)} $. 
By the monotone convergence theorem, we see that 
\begin{align}
\int_{-\infty }^{\infty } \psi(x) \exp \rbra{- \frac{x^2}{2 \sigma_n}} \d x 
\tend{}{n \to \infty } 
\int_{-\infty }^{\infty } \psi(x) \exp \rbra{- \frac{x^2}{2 \sigma}} \d x . 
\label{}
\end{align}
(The limit may possibly be infinite.) 
Since the Wiener integral for Brownian motion is Gaussian, we obtain 
\begin{align}
W \sbra{ \psi \rbra{ \int_0^{\infty } f_n(s) \d X_s } } 
\tend{}{n \to \infty } 
W \sbra{ \psi \rbra{ \int_0^{\infty } f(s) \d X_s } } . 
\label{}
\end{align}
Therefore we obtain \eqref{eq: FHY ineq} by Fatou's lemma and by Theorem \ref{thm: FHY2}. 
\end{proof}

Note that 
\begin{align}
R^+[X_t] 
= \sqrt{ \frac{2}{\pi} } \int_0^t \frac{\d s}{\sqrt{s}} 
, \quad t \ge 0 . 
\label{}
\end{align}

\begin{Lem}
Let $ f \in L^2(\d s) \cap L^1(\frac{\d s}{\sqrt{s}}) $. 
Then, under $ R^+ $, it holds that 
\begin{align}
\int_0^{\infty } f(s) \d X_s 
= \int_0^{\infty } f(s) \d \hat{X}_s 
+ \sqrt{\frac{2}{\pi}} \int_0^{\infty } f(s) \frac{\d s}{\sqrt{s}} 
\label{eq: Wint uncentered}
\end{align}
where 
the Wiener integral in the left hand side 
has been defined in \eqref{eq: first def}. 
\end{Lem}

\begin{proof}
It is obvious that the equality \eqref{eq: Wint uncentered} holds 
in the case where $ f $ is a step function. 
In the general case, we obtain \eqref{eq: Wint uncentered} 
by approximating $ f $ by step functions both in $ L^2(\d s) $ and in $ L^1(\frac{\d s}{\sqrt{s}}) $. 
\end{proof}

\begin{Rem} \label{rem: uncentered}
We cannot dispense with the assumption that $ f \in L^1(\frac{\d s}{\sqrt{s}}) $; 
in fact, if 
\begin{align}
f(s) = \frac{1}{\sqrt{s} \log s} 1_{(2,\infty )}(s) , 
\label{}
\end{align}
then 
$ f \in L^2(\d s) $ so that $ \int_0^{\infty } f(s) \d \hat{X}_s $ exists, 
while the integral $ \int_0^{\infty } f(s) \frac{\d s}{\sqrt{s}} $ diverges. 
See \cite{JY} for a very similar discussion. 
\end{Rem}

\section{Wiener integral under the $ \sigma $-finite measure} \label{sec: Wint sW}

\subsection{A limit theorem for last exit time}

Note that the last exit time from 0 up to time $ t $, denoted by 
$ g^{(t)}(X) = \sup \{ s \le t : X_s=0 \} $, 
has, under $ W $, the arcsine law: 
\begin{align}
W(g^{(t)}(X) \in \d u) = \frac{\d u}{\pi \sqrt{u(t-u)}} . 
\label{}
\end{align}
We need the following limit theorem. 

\begin{Thm} \label{lem L2}
Let $ \varphi $ be a non-negative non-increasing function on $ (0,\infty ) $ such that 
$ \varphi \in L^1(\frac{\d u}{\sqrt{u}}) $. 
Then it holds that 
\begin{align}
\lim_{t \to \infty } \sqrt{t} \int_0^t \varphi(u) \frac{\d u}{\sqrt{u(t-u)}} 
= \int_0^{\infty } \varphi(u) \frac{\d u}{\sqrt{u}} ; 
\label{eq: limit theorem}
\end{align}
in other words, 
\begin{align}
\lim_{t \to \infty } \sqrt{\frac{\pi t}{2}} W \sbra{ \varphi(g^{(t)}(X)) } 
= \sW \sbra{\varphi(g(X))} . 
\label{}
\end{align}
\end{Thm}

\begin{Rem}
We have a counterexample (see \cite[Example 6.1]{YYY}) 
if we omit the non-increasingness assumption. 
Theorem \ref{lem L2} is a special case of \cite[Lemma 6.3]{YYY}, 
which plays an important role in penalisation problems. 
\end{Rem}

\begin{proof}[Proof of Theorem \ref{lem L2}.]
It suffices to show that 
\begin{align}
\lim_{t \to \infty } 
\int_0^t \varphi(u) \rbra{ \sqrt{ \frac{t}{t-u} } - 1 } \frac{\d u}{\sqrt{u}} = 0 . 
\label{}
\end{align}
Since $ \sqrt{a}-\sqrt{b} \le \sqrt{a-b} $ for $ a \ge b \ge 0 $, 
it suffices to show that 
\begin{align}
\lim_{t \to \infty } 
\int_0^t \varphi(u) \frac{\d u}{\sqrt{t-u}} = 0 . 
\label{}
\end{align}

We note that 
\begin{align}
\int_{t/4}^t \varphi(u) \frac{\d u}{\sqrt{u}} 
\ge \varphi(t) \int_{t/4}^t \frac{\d u}{\sqrt{u}} 
= \sqrt{t} \varphi(t) . 
\label{}
\end{align}
Hence it follows from the assumption $ \varphi \in L^1(\frac{\d u}{\sqrt{u}}) $ that 
$ \sqrt{t} \varphi(t) \to 0 $ as $ t \to \infty $. 

Let $ 0<\eps<1 $. First, we have 
\begin{align}
\int_{\eps t}^t \varphi(u) \frac{\d u}{\sqrt{t-u}} 
\le& \varphi(\eps t) \int_{\eps t}^t \frac{\d u}{\sqrt{t-u}} 
\label{} \\
=& 2 \sqrt{\frac{1-\eps}{\eps}} \cbra{ \sqrt{\eps t} \varphi(\eps t) } 
\tend{}{t \to \infty } 0 . 
\label{}
\end{align}
Second, we have 
\begin{align}
\int_0^{\eps t} \varphi(u) \frac{\d u}{\sqrt{t-u}} 
\le& 
\int_0^{\eps t} \varphi(u) \frac{\d u}{\sqrt{(u/\eps)-u}} 
\label{} \\
=& 
\frac{1}{\sqrt{(1/\eps)-1}} \int_0^{\eps t} \varphi(u) \frac{\d u}{\sqrt{u}} . 
\label{}
\end{align}
Thus we obtain 
\begin{align}
\limsup_{t \to \infty } 
\int_0^t \varphi(u) \frac{\d u}{\sqrt{t-u}} 
\le \frac{1}{\sqrt{(1/\eps)-1}} \int_0^{\infty } \varphi(u) \frac{\d u}{\sqrt{u}} . 
\label{}
\end{align}
Letting $ \eps \to 0+ $, we obtain the desired result. 
\end{proof}

We shall utilize the following lemma. 

\begin{Lem} \label{lem L1}
Let $ \varphi $ be a non-negative non-increasing function on $ (0,\infty ) $ such that 
$ \varphi \in L^1(\frac{\d u}{\sqrt{u}}) $. 
Suppose that $ \varphi(0+) $ is finite. 
Then there exist two constants $ c_0 $ and $ C_0 $ such that, 
for any Borel function $ f $, it holds that 
\begin{align}
c_0 \int_0^{\infty } |f(s)| \frac{\d s}{1 + \sqrt{s}} 
\le 
\int_0^{\infty } \frac{\d u}{\sqrt{u}} \varphi(u) 
\int_0^{\infty } |f(s+u)| \frac{\d s}{\sqrt{s}} 
\le C_0 \int_0^{\infty } |f(s)| \frac{\d s}{1 + \sqrt{s}} . 
\label{}
\end{align}
In other words, the norm $ \| \cdot \|_{\varphi} $ defined by 
\begin{align}
\| f \|_{\varphi} = \int_0^{\infty } \frac{\d u}{\sqrt{u}} \varphi(u) 
\int_0^{\infty } |f(s+u)| \frac{\d s}{\sqrt{s}} 
\label{}
\end{align}
is equivalent to the norm $ \| \cdot \|_{L^1(\frac{\d s}{1+\sqrt{s}})} $. 
In particular, if $ f \in L^1(\frac{\d s}{1 + \sqrt{s}}) $, then it holds that 
\begin{align}
\int_0^{\infty } |f(s+u)| \frac{\d s}{\sqrt{s}} < \infty 
\quad \text{for a.e. $ u \in [0,\infty ) $} . 
\label{}
\end{align}
\end{Lem}

\begin{proof}
Changing the order of integration, we have 
\begin{align}
\| f \|_{\varphi} 
=& \int_0^{\infty } \frac{\d u}{\sqrt{u}} \varphi(u) 
\int_u^{\infty } |f(s)| \frac{\d s}{\sqrt{s-u}} 
\label{} \\
=& \int_0^{\infty } \d s |f(s)| \int_0^s \varphi(u) \frac{\d u}{\sqrt{u(s-u)}} . 
\label{}
\end{align}
Applying Theorem \ref{lem L2}, we see that 
\begin{align}
\int_0^s \varphi(u) \frac{\d u}{\sqrt{u(s-u)}} 
\sim 
\begin{cases}
(1/\sqrt{s}) \int_0^{\infty } \varphi(u) \frac{\d u}{\sqrt{u}} 
& \text{as $ s \to \infty $}, \\
\varphi(0+) \pi & \text{as $ s \to 0+ $}. 
\end{cases}
\label{}
\end{align}
Thus we obtain the desired result. 
\end{proof}

\subsection{Approximation theorems}

Let $ \varphi $ be a non-negative non-increasing function on $ (0,\infty ) $ such that 
$ \varphi \in L^1(\frac{\d u}{\sqrt{u}}) $ 
and that $ \varphi(0+) $ is finite. 
We are mainly interested in the measure $ \frac{\d u}{\sqrt{2 \pi u}} \Pi^{(u)} \bullet R $, 
but it will be more convenient to work with the finite measure 
\begin{align}
\mu_{\varphi}(\d u \times \d X) 
= \frac{\d u}{C_{\varphi} \sqrt{u}} \varphi(u) \rbra{ \Pi^{(u)} \bullet R }(\d X) 
\label{}
\end{align}
where $ C_{\varphi} = \int_0^{\infty } \varphi(u) \frac{\d u}{\sqrt{u}} $. 
Note that 
\begin{align}
\int F(u,X) \mu_{\varphi}(\d u \times \d X) 
= \frac{1}{\sW[\varphi(g(X))]} \int_{\Omega} \sW(\d X) \sbra{ \varphi(g(X)) F(g(X),X) } 
\label{}
\end{align}
for any non-negative measurable function $ F $ on $ [0,\infty ) \times \Omega $. 
For simplicity, we may choose $ \varphi(u) = \e^{-u} $ 
and write $ \mu $ for $ \mu_{\varphi} $. 

For $ f \in \cS $ and the coordinate process $ X $, we define 
\begin{align}
J^{(1)}(f;u,X^{(u)}) = \int_0^u f(s) \d X_s 
\quad \text{and} \quad 
J^{(2)}(f;u,\theta_u X) = \int_0^{\infty } f(s+u) \d (\theta_u X)_s . 
\label{}
\end{align}

\begin{Prop} \label{thm: approx--}
Let $ f \in \cS $. Then it holds that 
\begin{align}
\mu \sbra{ |J^{(1)}(f;u,X^{(u)})| } \le \| f \|_{L^2(\d s)} 
\label{eq: mu prop part1}
\end{align}
and that there exists an absolute constant $ C_1 $ such that 
\begin{align}
\mu \sbra{ |J^{(2)}(f;u,\theta_u X)| } 
\le C_1 \cbra{ \| f \|_{L^2(\d s)} 
+ \| f \|_{L^1(\frac{\d s}{1 + \sqrt{s}})} } . 
\label{eq: mu prop part2}
\end{align}
\end{Prop}

\begin{proof}
By definition, we have 
\begin{align}
\mu \sbra{ \absol{ \int_0^u f(s) \d X_s } } 
= \int_0^{\infty } \frac{\d u}{\sqrt{\pi u}} \e^{-u} 
\left\| \int_0^u f(s) \d X_s \right\|_{L^1(\Pi^{(u)})} 
\label{eq: mu prop 1}
\end{align}
By the Schwarz inequality and by the It\^o isometry \eqref{eq: Wint Pi}, we have 
\begin{align}
\text{\eqref{eq: mu prop 1}} 
\le \int_0^{\infty } \frac{\d u}{\sqrt{\pi u}} \e^{-u} 
\left\| \int_0^u f(s) \d X_s \right\|_{L^2(\Pi^{(u)})} 
\le \| \pi_uf \|_{L^2([0,u],\d s)} \le \| f \|_{L^2(\d s)} . 
\label{}
\end{align}
This proves \eqref{eq: mu prop part1}. 

By definition, we have 
\begin{align}
\mu \sbra{ \absol{ \int_0^{\infty } f(s+u) \d X_{s+u} } } 
= \int_0^{\infty } \frac{\d u}{\sqrt{\pi u}} \e^{-u} 
R \sbra{ \absol{ \int_0^{\infty } f(s+u) \d X_s } } . 
\label{eq: approx-1}
\end{align}
Using \eqref{eq: first def}, we have 
\begin{align}
R \sbra{ \absol{ \int_0^{\infty } f(s+u) \d X_s } } 
\le& 
R \sbra{ \absol{ \int_0^{\infty } f(s+u) \d B_s } } 
+ \int_0^{\infty } |f(s+u)| R \sbra{ \frac{1}{X_s} } \d s 
\label{} \\
\le& 
\| f \|_{L^2(\d s)} 
+ \sqrt{\frac{2}{\pi}} \int_0^{\infty } |f(s+u)| \frac{\d s}{\sqrt{s}} . 
\label{}
\end{align}
Hence we have 
\begin{align}
\text{\eqref{eq: approx-1}} 
\le \| f \|_{L^2(\d s)} 
+ \frac{\sqrt{2}}{\pi} \| f \|_{\varphi} 
\label{}
\end{align}
where $ \varphi(u) = \e^{-u} $. 
By Lemma \ref{lem L1}, we obtain \eqref{eq: mu prop part2}. 
The proof is now complete. 
\end{proof}

\begin{Thm} \label{thm: approx-}
Let $ f \in L^2(\d s) \cap L^1(\frac{\d s}{1 + \sqrt{s}}) $. 
Suppose that a sequence $ \{ f_n \} \subset \cS $ 
approximates $ f $ both in $ L^2(\d s) $ and in $ L^1(\frac{\d s}{1 + \sqrt{s}}) $. 
Then there exist jointly measurable functionals 
$ J^{(1)}(f;u,X^{(u)}) $ and $ J^{(2)}(f;u,\theta_u X) $ such that 
\begin{align}
J^{(1)}(f_n;u,X^{(u)}) 
\tend{}{n \to \infty } 
J^{(1)}(f;u,X^{(u)}) 
\quad \text{in $ \mu $-probability}
\label{eq: random vector}
\end{align}
and 
\begin{align}
J^{(2)}(f_n;u,\theta_u X) 
\tend{}{n \to \infty } 
J^{(2)}(f;u,\theta_u X) 
\quad \text{in $ \mu $-probability}. 
\label{eq: random vector2}
\end{align}
The limit functionals do not depend on the choice of the approximating sequence. 
Moreover, it holds for a.e. $ u \in [0,\infty ) $ 
and for $ (\Pi^{(u)} \bullet R)(\d X) $-a.e. $ X $ that 
\begin{align}
J^{(1)}(f;u,X^{(u)}) 
= \int_0^u f(s) \d X_s 
\quad \text{and} \quad 
J^{(2)}(f_n;u,\theta_u X) 
= \int_0^{\infty } f(s+u) \d (\theta_u X)_s . 
\label{eq: J1J2}
\end{align}
\end{Thm}

\begin{proof}
This is obvious by Proposition \ref{thm: approx--}. 
\end{proof}

Now we proceed to prove Theorem \ref{thm: approx} and Theorem \ref{thm: approx+} 
at the same time. 

\begin{proof}[Proof of Theorem \ref{thm: approx} and Theorem \ref{thm: approx+}.]
Let $ \{ f_n \} $ be a sequence of step functions 
such that $ \{ f_n \} $ approximates 
$ f $ both in $ L^2(\d s) $ and in $ L^1(\frac{\d s}{1 + \sqrt{s}}) $. 
Since $ f_n $ is a step function, we have 
\begin{align}
\int_0^{\infty } f_n(s) \d X_s 
= I(f_n;g(X),X) 
\label{}
\end{align}
where $ I(f_n;u,X) = J^{(1)}(f_n;u,X^{(u)}) + J^{(2)}(f_n;u,\theta_u X) $. 
Thus Theorem \ref{thm: approx-} shows that 
\begin{align}
\int_0^{\infty } f_n(s) \d X_s 
\tend{}{n \to \infty } 
I(f;g(X),X) 
\quad \text{in $ \sW^G $-probability}
\label{}
\end{align}
where $ I(f;u,X) = J^{(1)}(f;u,X^{(u)}) + J^{(2)}(f;u,\theta_u X) $ 
and where $ G(X)=\e^{-g(X)} \in L^1_+(\sW) $. 
By (ii) of Proposition \ref{prop: loc in meas}, 
we obtain 
\begin{align}
\int_0^{\infty } f_n(s) \d X_s 
\tend{}{n \to \infty } 
I(f;g(X),X) 
\quad \text{locally in $ \sW $-measure}
\label{}
\end{align}
where 
\begin{align}
I(f;u,X) 
= J^{(1)}(f;u,X^{(u)}) + J^{(2)}(f;u,\theta_u X) . 
\label{}
\end{align}
This completes the proof. 
\end{proof}

\subsection{Continuous modification}

Let us prove Theorem \ref{thm: conti mod}. 

\begin{proof}[Proof of Theorem \ref{thm: conti mod}.]
For the coordinate process $ X $ and for $ u=g(X) $, we define 
\begin{align}
\hat{(\theta_u X)}_t = (\theta_u X)_t - \sqrt{\frac{2}{\pi}} \int_0^t \frac{\d s}{\sqrt{s}} 
\quad \text{for $ t \ge 0 $}. 
\label{}
\end{align}

Let $ f \in L^2([0,T],\d s) $. Then we may define 
\begin{align}
J^{(3)}(f;u,\theta_u X) = \int_0^{\infty } f(s+u) \d \hat{(\theta_u X)}_s 
\quad \text{$ \mu $-a.s.} 
\label{}
\end{align}
Applying Theorem \ref{thm: FHY} for $ \psi(x)=x^4 $ 
and then using the Gaussian property of $ \{ (X_t),W \} $, 
we see that 
\begin{align}
\mu \sbra{ \absol{ J^{(3)}(f;u,\theta_u X) }^4 } 
\le \int_0^{\infty } \frac{\d u}{\sqrt{\pi u}} \e^{-u} 
R \sbra{ \absol{ J^{(3)}(f;u,\theta_u X) }^4 } 
\le 3 \| f \|_{L^2(\d s)}^2 . 
\label{}
\end{align}

For $ t \in [0,T] $, we write $ f_t = f 1_{[0,t)} $. 
Set 
\begin{align}
M(t) = t + \int_0^t |f(s)|^2 \d s . 
\label{}
\end{align}
Since the function $ v=M(t) $ is continuous and strictly-increasing, 
there exists its continuous inverse $ t=L(v) $. 
Then, for $ 0 \le v_1 < v_2 \le T $, it holds that 
\begin{align}
& \mu \sbra{ \absol{ J^{(3)}(f_{L(v_2)};u,\theta_u X) 
- J^{(3)}(f_{L(v_1)};u,\theta_u X) }^4 } 
\label{} \\
\le& 3 \rbra{ \int_{L(v_1)}^{L(v_2)} |f(s)|^2 \d s }^2 
\le 3 |v_2-v_1|^2 . 
\label{}
\end{align}
From this inequality, 
we appeal to Kolmogorov's continuity theorem, 
and we see that 
there exists a process $ (K^{(3)}_v(f;u,\theta_u X):v \in [0,M(T)]) $ 
which is a $ \mu $-a.s. continuous modification 
of $ \{ J^{(3)}(f_{L(v)};u,\theta_u X): v \in [0,M(T)] \} $. 
In the same way as above, 
we may construct a continuous process $ \{ K^{(1)}_v(f;u,X^{(u)}):v \in [0,M(T)]) \} $ 
which is a $ \mu $-a.s. continuous modification 
of $ \{ J^{(1)}(f_{L(v)};u,X^{(u)}): v \in [0,M(T)] \} $. 

Set 
\begin{align}
U = \cbra{ u \in [0,\infty ) : \int_0^{\infty } |f(s+u)| \frac{\d s}{\sqrt{s}} < \infty } . 
\label{eq: def of U}
\end{align}
By Lemma \ref{lem L1}, we see that $ U^c $ has Lebesgue measure zero. 
For $ u \in U $ and $ t \in [0,T] $, we define 
\begin{align}
I_t(f;u,X) 
= K^{(1)}_{M(t)}(f;u,X^{(u)}) + K^{(3)}_{M(t)}(f;u,\theta_u X) 
+ \sqrt{\frac{2}{\pi}} \int_0^{\infty } f_t(s+u) \frac{\d s}{\sqrt{s}} 
\label{}
\end{align}
and, for $ u \notin U $ and $ t \in [0,T] $, we define $ I_t(f;u,X) = 0 $. 
Therefore we conclude that 
the resulting process $ \{ I_t(f;u,X):t \in [0,T] \} $ 
is as desired. 
\end{proof}

{\bf Acknowledgements.} 
The author would like to thank Professors Marc Yor and Tadahisa Funaki 
for their fruitful comments.

\bibliographystyle{plain}

\end{document}